\documentclass[12pt]{amsart}

\usepackage{CJK,CJKnumb}
\usepackage{amsfonts}
\usepackage{amsthm}
\usepackage{geometry}
\usepackage{mathrsfs}
\usepackage{amsmath}
\usepackage{amssymb}
\usepackage{amsfonts}
\usepackage[percent]{overpic}
\usepackage{bm}
\usepackage[all]{xy}
\usepackage{graphicx}
\usepackage{subfigure}
\usepackage{latexsym}
\usepackage[colorlinks, linkcolor=blue, anchorcolor=bule, citecolor=blue]{hyperref}
\usepackage{epigraph}
\usepackage{hyperref}
\usepackage{fancyhdr}
\usepackage[T1]{fontenc}
\usepackage[english]{babel}
\usepackage{comment}

\numberwithin{equation}{section}
\usepackage{color}

\newtheorem{theorem}{Theorem}[section]
\newtheorem{lemma}{Lemma}[section]
\theoremstyle{remark}
\newtheorem{remark}{Remark}[section]

\newtheorem{definition}{Definition}[section]

\newtheorem*{ack}{Acknowledgement}

\def\re{\operatorname{Re}}
\def\im{\operatorname{Im}}
\def\area{\operatorname{area}}
\def\dens{\operatorname{dens}}
\def\meas{\operatorname{meas}}

\def\c{\operatorname{\mathbb C}}

\def\j{\operatorname{\mathcal{J}}}
\def\f{\operatorname{\mathcal{F}}}
\def\i{\operatorname{\mathcal{I}}}
\def\b{\operatorname{\mathcal{B}}}

\begin{document}
\title[Lebesgue measure of escaping sets]{Lebesgue measure of escaping sets of entire functions}
\author{Weiwei Cui}

\address{Mathematiches Seminar, Christian-Albrechts-Universit\"at zu Kiel, Ludewig-Meyn-Str. 4, 24098 Kiel, Germany.}
\email{cui@math.uni-kiel.de}


\begin{abstract}
For a transcendental entire function $f$ of finite order in the Eremen\-ko-Lyubich class $\b$, we give conditions under which the Lebesgue measure of the escaping set $\i(f)$ of $f$ is zero. This complements the recent work of Aspenberg and Bergweiler, in which they give conditions on entire functions in the same class with escaping sets of positive Lebesgue measure. We will construct an entire function in the Eremenko-Lyubich class to show that the condition given by Aspenberg and Bergweiler is essentially sharp. Furthermore, we adapt our idea of proof to certain infinite-order entire functions. Under some restrictions to the growth of these entire functions, we show that the escaping sets have zero Lebesgue measure. This generalizes a result of Eremenko and Lyubich.
\end{abstract}

\maketitle

\section{Introduction and main results}

The \textit{escaping set} $\i(f)$ of a transcendental entire function $f$ is defined as the set of all points in $\c$ where the iterates $f^n$ of $f$ tend to $\infty$ as $n\to\infty$. Even though the investigation of the Julia set and Fatou set of a transcendental entire function started from Fatou \cite{fatou4}, a thorough study of the escaping set was not undertaken until Eremenko \cite{eremenko3}. Various structures of escaping sets have been investigated, from either topological, geometrical or measure-and-dimension-theoretical points of view. Our aim here is to study escaping sets of transcendental entire functions in view of their two-dimensional Lebesgue measure. 

Recall that the set $S(f)$ of singular values of $f$ is the closure of the set of all critical and asymptotic values of $f$ in $\c$. An entire function $f$ belongs to the \textit{Eremenko-Lyubich class} $\b$ if $S(f)$ is bounded. Dynamics of entire functions in class $\b$ have attracted a lot of interests in recent years, see, for instance, \cite{eremenko2, rempe2, bergweiler6, rempe16}. Functions in this class include, for instance, $\lambda e^z$ and $\sin(\alpha z+\beta)$, where $\lambda \in\c\setminus\{0\}, \alpha, \beta\in\c, \alpha\neq 0$. 

McMullen  showed that $\i(\sin(\alpha z+\beta))$ has positive measure \cite{mcmullen11}. This result was substantially generalized by Aspenberg and Bergweiler \cite{aspenberg1} to functions in class $\b$ with some control on the growth of the functions. To formulate their conditions, consider the function
$$E_{\beta}(z)=e^{\beta z} ~~~\text{~~~for~~~}~\beta\in(0,{1}/{e}).$$
This is a real entire function with one fixed point $\xi$ which is repelling and has multiplier $\lambda=\beta\xi>1$. Around this repelling fixed point, Schr\"oder's functional equation
\begin{equation}\label{s11}
\Phi\left(E_{\beta}(z)\right)=\lambda \Phi(z)
\end{equation}
has a unique local holomorphic solution $\Phi$ normalized by $\Phi(\xi)=0$ and $\Phi'(\xi)=1$. The function $\Phi$ has a continuation to the positive real axis such that (\ref{s11}) still holds for $z\in[\xi, \infty)$. Note that $\Phi$ is increasing on the real axis and $\lim_{x\to\infty}\Phi(x)=\infty$. Moreover, the function $\Phi$ tends to $\infty$ slower than any iterate of the logarithm, in other words, for all $m\in\mathbb{N}$ we have
\begin{equation}\label{phi}
\lim_{x\to\infty}\frac{\Phi(x)}{\log^m x}=0.
\end{equation}
Here $\log^m$ denotes the $m$-th iterate of the logarithm. 

Now the result of Aspenberg and Bergweiler can be stated as follows. Let $f$ be an entire function for which $\{z: |f(z)|>R\}$ has $N$ components for some $R>0$, and suppose that
\begin{equation}\label{abb}
\log\log M(r,f)\leq \left(\frac{N}{2}+\frac{1}{\Phi(r)}\right)\log r
\end{equation}
for large $r$. Then the escaping set of $f$ has positive Lebesgue measure. We call condition (\ref{abb}) the {\textit{Aspenberg-Bergweiler condition}. The Denjoy-Carleman-Ahlfors theorem \cite[p.173]{goldbergmero} gives
\begin{equation}\label{abb11}
\log\log M(r,f)\geq \frac{N}{2}\log r-\mathcal{O}(1).
\end{equation}
Therefore, the Aspenberg-Bergweiler condition ensures that such an entire function grows only at most slightly faster than guaranteed by the Denjoy-Carleman-Ahlfors theorem. Aspenberg and Bergweiler also used the Mittag-Leffler function to show that $\Phi(r)$ in (\ref{abb}) can not be replaced by a positive constant.

We complement the result of Aspenberg and Bergweiler with the following theorem.

\begin{theorem}\label{main result}
There exists an entire function $f$ in class $\b$ with
\begin{equation}\label{examplec}
\log\log M(r,f)\leq \left(\frac{1}{2}+\frac{1}{\log \Phi(r)}\right)\log r +\mathcal{O}(1),
\end{equation}
for which the escaping set has Lebesgue measure zero.
\end{theorem}

It may appear from comparing that there is still a gap between the Aspenberg-Bergweiler condition and Theorem \ref{main result}. However, what Aspenberg and Bergweiler actually use is that
\begin{equation}\label{sum1}
\sum_{k=1}^{\infty}\frac{1}{\Phi\left(E^{k}(0)\right)}=\sum_{k=1}^{\infty}\frac{1}{\lambda^k \Phi(0)}<\infty.
\end{equation}
They work with $E_{\beta}$ instead of $E$, but this is irrelevant, see the proof in Lemma \ref{infinite} below. Noting that
\begin{equation}\label{sum111}
\sum_{k=1}^{\infty}\frac{1}{\left(\log\Phi\left(E^{k}(0)\right)\right)^{1+\delta}}=\sum_{k=1}^{\infty} \frac{1}{\left(k\log\lambda + \log\Phi(0)\right)^{1+\delta}}<\infty,
\end{equation}
if $\delta>0$, we see that minor modifications of the proof yield that their result holds with \eqref{abb} replaced by
\begin{equation}
\log\log M(r,f)\leq \left(\frac{N}{2}+\frac{1}{\left(\log\Phi(r)\right)^{1+\delta}}\right)\log r.
\end{equation}
Our result shows that this is not the case for $\delta=0$.

In our proof we will show that $\rho(r):=1/2+1/\log\Phi(r)$ is a proximate order; see Section 2 for the definition of proximate orders. In particular, we show that $\varepsilon(r):=1/\log\Phi(r)$ satisfies
\begin{equation}\label{prvar}
\lim_{r\to\infty}\varepsilon'(r)r\log r =0.
\end{equation}
Our proof of Theorem \ref{main result} also yields the following result.

\begin{theorem}\label{corothm1}
Suppose that $\varepsilon(r)$ satisfies \eqref{prvar} and
\begin{equation}\label{prvar2}
\sum_{k=1}^{\infty}\varepsilon\left(E^{k}(0)\right)=\infty.
\end{equation}
Then there exists an entire function $f\in\b$ with 
\begin{equation}\label{examplec}
\log\log M(r,f)\leq \left(\frac{1}{2}+\varepsilon(r)\right)\log r +\mathcal{O}(1),
\end{equation}
for which the escaping set has Lebesgue measure zero.
\end{theorem}

As explained above, the method of Aspenberg and Bergweiler shows that \eqref{prvar2} is sharp.

\smallskip
Theorem \ref{main result} and Theorem \ref{corothm1} will follow from Theorem \ref{thm1.2} below which may be of independent interest. We first recall some definitions.

A transcendental entire function is said to be \emph{hyperbolic}, if $f\in\b$ and every element of $S(f)$ belongs to the basin of some attracting periodic cycles of $f$. If $f$ is hyperbolic and $\f(f)$ is connected, then $f$ is of \emph{disjoint type}; see Section $2$ for more details. The order of growth of an entire function $f$ is defined as
\begin{equation}\label{order}
\rho(f)= \limsup_{r\to\infty}\dfrac{\log\log M(r,f)}{\log r}.
\end{equation}
 Here, $M(r,f)$ is the maximum modulus of $f$, defined by 
$$M(r,f)=\max_{|z|=r}\left|f(z)\right|.$$

Given an entire function $f\in\b$ with $S(f)\subset D(0,r_0)$, put
\begin{equation}\label{theta1}
\theta(r):=\meas\left\{t\in[0, 2\pi]: |f(re^{it})|<r_0   \right\}.
\end{equation}
Here $\meas$ denotes the one-dimensional Lebesgue measure. Now our result can be stated as follows.

\begin{theorem}\label{thm1.2}
Let $f\in\b$ be of finite order. Let $\theta$ be as above. Suppose that $\theta(r)\geq \theta_0 (r)$ for large $r>0$, where $\theta_{0}(r)$ is decreasing and satisfies
\begin{equation}\label{theta-condition}
\sum_{k=1}^{\infty}\theta_{0}\left(E^k(0)\right)=\infty.
\end{equation}
Then $\area\i(f)=0$. If, in addition, $f$ is of disjoint type, then $\area\j(f)=0$.
\end{theorem}

Some additional hypothesis such as disjoint type is necessary to conclude that $\area\j(f)=0$ in Theorem \ref{thm1.2}. In fact, it was shown by McMullen \cite{mcmullen11} and also follows from Theorem \ref{thm1.2} that $\area\i(e^z)=0$. However, $\j(e^z)=\c$ by a result of Misiurewicz \cite{misiurewicz2}.

Note that for any entire function $f\in\b$, the function
$$f_{\lambda}: \c\to\c;~~~~z\mapsto \lambda f(z)$$
is also in class $\b$ for any $\lambda\in\c\setminus\{0\}$. If we choose $\lambda$ to be sufficiently small, then $f_{\lambda}$ is of disjoint type. We first prove Theorem \ref{thm1.2} for disjoint type functions. In this case we have the conclusion that the area of the Julia set is zero. To transfer to entire functions without being disjoint type, we apply a result of Rempe \cite[Theorem 1.1]{rempe8}. In this situation we see that the area of escaping set is zero.

\smallskip
Now we sketch the construction of the entire function in Theorem \ref{main result}. The main idea is to consider a function $f$ which is a canonical product with zeros distributed along the positive real axis and then control the asymptotic behavior of $f$ outside of a small curvilinear sector containing the positive real axis. By applying the Denjoy-Carleman-Ahlfors theorem, we show that the function is in fact bounded in this sector. We will see that all the critical points of $f$ are contained in the above sector and hence the set of all critical values is bounded. For this function the only possible asymptotic value is $0$, so it is in the Eremenko-Lyubich class $\b$. \hyperref[them1.2]{Theorem \ref{thm1.2}} then implies that the escaping set of such a function has measure zero.

\medskip
In 1992, Eremenko and Lyubich \cite{eremenko2} also gave a condition under which the escaping set of an entire function in class $\mathcal B$ has zero Lebesgue measure. Their condition was formulated in terms of $\theta(r)$ defined by \eqref{theta1}. More precisely, they considered transcendental entire functions satisfying
\begin{equation}\label{EL-condition}
\liminf_{r\to\infty}\frac{1}{\log r}\int_1^r \theta(t)\frac{dt}{t}>0.
\end{equation}
They show that escaping sets of transcendental entire functions in class $\b$ satisfying (\ref{EL-condition}) have zero Lebesgue measure. We call condition (\ref{EL-condition}) the {\textit{Eremenko-Lyubich condition}}. This condition implies that for a transcendental entire function $f\in\b$, if $f$ is bounded in a sector, then the area of the escaping set is zero (in case of finite order, this also follows from Theorem \ref{thm1.2}). An explicit example that illustrates the power of the Eremenko-Lyubich condition is the above-mentioned Mittag-Leffler function.

Eremenko and Lyubich show that (\ref{EL-condition}) is satisfied if $f$ is an entire function of finite order for which the inverse $f^{-1}$ has a logarithmic singularity; see Section $2$ for the classification of singularities. Thus if, in addition, $f$ belongs to class $\b$, then the escaping set of $f$ has measure zero. We show that this also holds for certain functions of infinite order.

\begin{theorem}\label{infinite-order1}
Let $f\in\b$ be a transcendental entire function and $r'>0$. Suppose that the inverse of $f$ has a direct singularity $a\in\mathbb{C}$. Suppose furthermore that $f$ satisfies
\begin{equation}\label{growwww}
\log\log M(r,f) \leq A(r)\log r
\end{equation}
for $r\geq r'$, for some continuous and increasing function $A: [r', \infty)\to\mathbb{R}$ satisfying $A(r)<\log r$ for large $r$ and
\begin{equation}\label{di}
\sum_{k=1}^{\infty} \dfrac{1}{A\left(E^{k}\left(0\right)\right)}=\infty.
\end{equation}
Then $\area \i(f)=0$.
\end{theorem}

The condition \eqref{di} implies that $A(r)<\log r$ is only a mild restriction.

Compare the theorem with the result of Eremenko-Lyubich we mentioned above. If the function $A(r)$ is bounded above by some constant, that is, $f$ is of finite order then the condition \eqref{di} is automatically satisfied. Then the statement of \hyperref[infinite-order1]{Theorem \ref{infinite-order1}} is yields that of Eremenko and Lyubich. In this sense, we can view this theorem as a generalization of Eremenko and Lyubich's result mentioned above.  Theorem \ref{infinite-order1} will follow from the following result.

\begin{theorem}\label{infinite-order2}
Let $f\in\b$ be a transcendental entire function and $r'>0$. Suppose that $f$ satisfies \eqref{growwww}
and
\begin{equation}\label{didi}
\frac{1}{\log r}\int_{r^c}^{r}\theta(t)\frac{dt}{t}\geq \frac{1}{A(r)}
\end{equation}
for a constant $c$ with $63/65\leq c <1$, and for some continuous and increasing function $A: [r', \infty)\to\mathbb{R}$ satisfying $A(r)<\log r$ for large $r$ and \eqref{di}. Then $\area \i(f)=0$.
\end{theorem}

We will prove this theorem for disjoint type entire functions satisfying the conditions in \hyperref[infinite-order2]{Theorem \ref{infinite-order2}}, and then apply the result of Rempe already quoted to obtain the result for all functions satisfying the hypotheses of the theorem.\\


\noindent{\bf\textit{Structure of the article.}} In section $2$ we will give some preliminaries that are required for the proof of the above theorems. Section $3$ is devoted to the proof of Theorem \ref{thm1.2}. Then we construct an entire function in Section $4$ which on the one hand satisfies the conditions in Theorem \ref{thm1.2} and on the other hand shows the essential sharpness of the Aspenberg-Bergweiler condition. Finally, section $5$ extends our methods in Theorem \ref{thm1.2} to the case of infinite-order entire functions in class $\b$, which generalizes the result of Eremenko and Lyubich as claimed.

\section{Preliminaries}

In this section, we give some definitions, notations and some basic dynamic properties of entire functions in class $\b$. Throughout we only consider transcendental entire functions. When we write $f$ we always mean such function.

\subsection{Some definitions and notations} For a transcendental entire function $f$, the \textit{Julia set} $\j(f)$ of $f$ is defined to be the set of all points in $\c$ where the iterates $f^n$ do not form a normal family in the sense of Montel. The complement $\f(f)$ of $\j(f)$ is called the \textit{Fatou set}. See \cite{bergweiler1} and \cite{schleicher3} for an introduction to transcendental dynamics.

A point $a\in\mathbb{C}$ is called a \textit{critical point} of $f$ if $f'(a)=0$, and $f(a)$ is called a \textit{critical value} of $f$. We say that $b\in\c\cup\{\infty\}$ is an \textit{asymptotic value} of $f$, if there exists a curve $\gamma$ tending to $\infty$ such that $f(\gamma)$ tends to $b$. A point $z$ is called a \textit{singularity} of $f^{-1}$, if it is an asymptotic value or a critical value of $f$. The set $S(f)$ of singular values, as mentioned before, is the closure of the set of all critical and asymptotic values of $f$ in $\c$.

Moreover, a singularity $z$ of $f^{-1}$ is called a \textit{direct singularity} if there exists a disk $D_{\chi}(z,r)$ with respect to the spherical metric such that $f$ omits the value $z$ in a component $V_r$ of $f^{-1}(D_{\chi}(z,r))$ for some $r>0$. In particular, a direct singularity is called \textit{logarithmic} if the restriction $f: V_r \to D_{\chi}(z,r)\setminus\{z\}$ is a universal covering for some $r>0$. The domain $V_r$ is called a \textit{direct tract} or \textit{logarithmic tract} over $z$, respectively. We say that an entire function $f$ is in the Eremenko-Lyubich class $\b$, if $S(f)$ is bounded.

\begin{definition}[Hyperbolicity and disjoint-type]
A transcendental entire function $f$ is said to be \emph{hyperbolic}, if $f\in\b$ and every element of $S(f)$ belongs to the basin of some attracting periodic cycles of $f$. If $f$ is hyperbolic and $\f(f)$ is connected, then $f$ is of \emph{disjoint type}.
\end{definition}

For the above definitions, we refer to \cite[Definition 1.1]{bergweiler6} and  \cite[Definition 1.1]{rempe9}. A discussion of the notion of hyperbolicity in the transcendental setting is given in \cite{rempe16}. We remark that, since every Fatou component of a hyperbolic entire function is simply connected \cite[Proposition 3]{eremenko2}, it follows that the Fatou set of a disjoint type entire function is simply connected.

Let $A, B\subset \mathbb{C}$ be measurable. Then the \textit{density} of $A$ in $B$ is defined to be
\begin{equation*}
\dens \left(A, B\right) = \dfrac{\area\left(A\cap B\right)}{\area B}.
\end{equation*}

\subsection{Logarithmic change of variables}

If $f\in \mathcal B$, then by definition we can find a constant, say $r_0 > 0$, such that all the singularities of $f$ lie in $\{z: |z|\leq r_0\}$. This implies that all components of $f^{-1}(\{z: |z|> r_0\})$ are logarithmic tracts over $\infty$. Without loss of generality, by choosing a suitable large constant $r_0$ we may assume that $|f(0)|\leq r_0$. For such functions, we can apply the \emph{logarithmic change of variable}, which was first introduced into transcendental dynamics by Eremenko and Lyubich \cite{eremenko2}. To describe this, we define
\begin{equation}\label{notation}
\begin{aligned}
A&=\{z\in\mathbb{C}:~|z|>r_0\},\\
U&=f^{-1}(A),\\
W&=\exp^{-1}(U),\\
H&=\{z\in \mathbb{C}:\re z>\log r_0\}.
\end{aligned}
\end{equation}
Eremenko and Lyubich \cite[Section 2]{eremenko2} proved that there exists $F: W\to H$ such that the following diagram commutes:
$$
\begin{xy}
(0, 15)*+{W}="a"; (20, 15)*+{H}="b"; %
(0, 0)*+{U}="c"; (20, 0)*+{A}="d"; %
{\ar "a"; "b"}? * !/_3mm/{F};%
{\ar "c"; "d"}? * !/^3mm/{f};%
{\ar "a"; "c"}? * !/^4mm/{\exp};%
{\ar "b"; "d"}? * !/_4mm/{\exp};%
\end{xy}
$$
Moreover, $F$ maps every component of $W$ biholomorphically onto $H$. We say that $F$ is obtained from $f$ by a logarithmic change of variables. 

In the following we shall use the estimate of the modulus of the derivative of $F$, which is given by Eremenko and Lyubich in \cite[Lemma 1]{eremenko2}.
\begin{lemma}[Expanding property]
Suppose $F$ is obtained through the logarithmic change of variable from $f$ as above, then
\begin{equation}\label{estimate of F'}
\left|F'(z)\right|\geq \dfrac{\re F(z)-\log r_0}{4\pi}
\end{equation}
for $z\in W$. 
\end{lemma}

A direct consequence of the expanding property is that $\i(f)\subset\j(f)$ for $f\in\b$ \cite{eremenko2}. Note that for $f$ entire, $\i(f)$ is non-empty \cite{eremenko3}.

\subsection{Quasiconformal equivalence near infinity} As we mentioned in the introduction, Theorem \ref{thm1.2} will be proved first for disjoint type functions. Then we use a result of Rempe \cite[Theorem 1.1]{rempe8} to transfer the result to entire functions not necessarily being of disjoint type. To formulate his result we need the notion of quasiconformal equivalence near infinity. We refer the reader to \cite{ahlfors8} for some basic background on quasiconformal mappings. Following \cite{rempe8}, we say that two entire functions $f, g\in\b$ are \textit{quasiconformally equivalent near $\infty$} if there exist quasiconformal mappings $\varphi, \psi : \c\to\c$ such that
\begin{equation}
\psi(f(z))=g(\varphi(z))
\end{equation}
whenever $|f(z)|$ or $|g(\varphi(z))|$ is large enough. The following result of Rempe roughly says that quasiconformal equivalence near $\infty$ implies quasiconformal conjugacy on some subset of the plane. More precisely, he shows

\begin{theorem}\label{rempe}
Let $f,g\in\b$ be quasiconformally equivalent near infinity. Then there exist $R>0$ and a quasiconformal map $\theta: \c\to\c$ such that
$$\theta\circ f=g\circ\theta~\text{~on~}~\j_{R}(f):=\left\{z\in\c: |f^{n}(z)|\geq R~\text{~for all~}~n\geq 1\right\}.$$
Futhermore, $\theta$ has zero dilatation on $\left\{z\in\j_{R}(f): |f^{n}(z)|\to\infty\right\}$.\\
\end{theorem}

\section{Proof of \hyperref[them1.2]{Theorem \ref{thm1.2}}}

To prove Theorem \ref{thm1.2}, we shall first prove the following technical version. Recall the notations in \eqref{notation} and \eqref{theta1}.

\begin{theorem}\label{main theorem}
Let $f\in\mathcal B$ be of finite order. Suppose that there exists $R_1$ with $R_1>\max\{2\log r_0, \log r_0 +64\pi\}$ such that
\begin{equation}\label{assume}
W\subset\left\{ z: \re z >R_1   \right\}.
\end{equation}
Suppose that $\theta(r)\geq \theta_0 (r)$ for large $r>0$, where $\theta_{0}(r)$ is decreasing and satisfies \eqref{theta-condition}.
Then $\area\j(f)=0$.
\end{theorem}

We will use the notations in (\ref{notation}) given in the introduction when applying a logarithmic change of variables. For simplicity, we let $r_0=e^R$ and hence $H=\{z\in\mathbb{C}: \re z > R\}$. Let $F$ be the function obtained from $f$ by using a logarithmic change of variables. 

We shall consider the following set
$$T=\left\{z: F^{n}(z)\in W,~\text{for all}~n\in\mathbb N_0\right\}.$$
Here $\mathbb{N}_0$ is the set of all non-negative integers. For disjoint type entire functions, the Fatou set of $f$ consists of a single immediate attracting basin, and 
$$\j(f)=\exp(T).$$
Moreover, the assumption that $f\in\b$ implies that $\i(f)\subset\j(f)$. And since exponential maps preserve sets of zero Lebesgue measure, to show that $\area \j(f)=0$ (and hence $\area \i(f)=0$) it suffices to show that $\area T=0$. In the following discussion, we will mainly concentrate on this set and prove that the Lebesgue measure of $T$ is zero. 

Define
$$T_{n}=\left\{z\in\c : F^{k}(z)\in W,~\text{for}~k=0,\dots,n\right\}$$
and
\begin{equation}\label{sss}
S=\mathbb{C}\setminus T_0 =  \mathbb{C}\setminus W.
\end{equation}
By definition of $T$ and $T_n$, we have
$$T=\bigcap_{n=0}^{\infty}T_n.$$


For $z_0\in\mathbb{C}$ and $r>0$, we use the notation $D(z_0,r)=\left\{z\in\mathbb{C}: \left|z-z_0\right| <r\right\}$ and in case of the unit disk we use $\mathbb D$.


The following is a Vitali type covering lemma which can be found in \cite[Lemma 4.8]{falconer1}. It holds for any bounded set in $\mathbb{R}^n$, but we only use it for sets in the complex plane $\c$.

\begin{lemma}\label{lemma 1}
Let $Q\subset\c$ be a bounded set and $r: Q \rightarrow \left(0, R\right]$ be a real positive function. Then there exists an at most countable subset $L$ of $Q$ such that
$$D\left(x,r(x)\right) \cap D\left(y,r(y)\right)=\emptyset~~\text{~for~}~~ x, y \in L, ~x\neq y,$$
and
$$\bigcup_{x\in Q}D\left(x,r(x)\right) \subset \bigcup_{x\in L}D\left(x,4r(x)\right).$$
\end{lemma}

The Koebe distortion theorem and the Koebe one quarter theorem are well-known, see \cite[Section 1.3]{pommerenke1}. We shall use the following version, which can be obtained by an easy argument from that given in \cite{pommerenke1}.

\begin{lemma}[Koebe's theorem]\label{lemma 2}
Let $f$ be a univalent function in $D(z_0,r)$ and let $0<\lambda<1$. Then
\begin{equation*}\label{estimate of value}
\dfrac{\lambda}{\left(1+\lambda \right)^2}\left|f'(z_0)\right|\leq \left| \dfrac{f(z)-f(z_0)}{z-z_0}\right| \leq \dfrac{\lambda}{\left(1-\lambda \right)^2}\left|f'(z_0)\right|
\end{equation*}
and
\begin{equation*}\label{estimate of derivative}
\dfrac{1-\lambda}{\left(1+\lambda \right)^3}\left|f'(z_0)\right|\leq \left| f'(z)\right| \leq \dfrac{1+\lambda}{\left(1-\lambda \right)^3}\left|f'(z_0)\right|
\end{equation*}
for $|z-z_0|\leq \lambda r$. Moreover,
\begin{equation*}\label{one-quarter theorem}
f\left(D(z_0,r)\right)\supset D\left( f(z_0), \frac{1}{4}|f'(z_0)|r  \right).
\end{equation*}
\end{lemma}

\bigskip
\begin{proof}[Proof of Theorem \ref{main theorem}]
First we put
\begin{equation}\label{cn}
c_n=\sum_{j=1}^{n}\frac{\tau R_1}{K^{j}}, ~~~~K=\frac{R_1-R}{4\pi} >1,
\end{equation}
where $\tau>0$ is some small constant to be determined later. 

For a point $w\in W$, we consider a sequence of squares centred at $w$ as follows:
\begin{equation}\label{sequs}
\begin{aligned}
P_{n}(w)=\left\{z \in \mathbb{C}: \left|\re(z -w)\right| \leq \frac{\re w}{64}-c_n,
\left|\im(z-w)\right|\leq \frac{\re w}{64}-c_n\right\}.
\end{aligned}
\end{equation}
We write $P_n$ instead of $P_{n}(w)$ for simplicity. For $z\in T_n$, we define
\begin{equation}\label{radius}
r_{n}(z)=\dfrac{\re F^{n}(z)}{|(F^n)'(z)|}.
\end{equation}
We show that there exists a countable subset $L_n \subset T_n \cap P_{n}$ satisfying the following conditions:\\

\begin{itemize}
\item[{(i)}] $\bigcup_{z\in T_n \cap P_{n}}D(z, \tau r_n (z))\subset \bigcup_{z\in L_n}D(z, 4\tau r_n (z))$;\\

\item[{(ii)}] $D(z_1, \tau r_n (z_1))\cap D(z_2, \tau r_n (z_2))=\emptyset~~\text{for distinct}~z_1, z_2 \in L_n$;\\

\item[{(iii)}] $D(z, \tau r_n (z))\subset P_{n-1},~~\text{for}~z\in L_{n};$\\

\item[{(iv)}] for each $z\in L_n$, the disk $D(z, \tau r_n (z))$ contains a compact subset $A_n(z)$ such that $F^{n}$ maps $A_n (z)$ bijectively onto a square $Q(z_n)$ centred at $z_n:=F^{n}(z)$ with sidelength $\re z_n/32$, that is,
\begin{equation}\label{square}
Q(z_n)=\left\{z \in \mathbb{C}: \left|\re (z - z_n)\right|\leq \frac{1}{64} \re z_n, ~\left|\im (z - z_n) \right|\leq \frac{1}{64} \re z_n \right\};
\end{equation}

\item[{(v)}] $D(z, \tau r_n (z))\subset T_{n-1},~~\text{for}~z\in L_{n}$.\\
\end{itemize}

The existence of $L_n$ satisfying (i) and (ii) follows from Lemma \ref{lemma 1}. To see that the conclusion (iii) holds, note that $z\in T_n$ and hence $F^{k}(z)\in W$ for $0\leq k\leq n$, which in particular means that $\re F^{n}(z)> R_1$. Moreover, it follows from \eqref{estimate of F'} and \eqref{cn} that
$$\left|F'(z)\right|\geq \frac{\re F(z)-R}{4\pi}\geq \frac{R_1 -R}{4\pi}=K.$$
Thus
\begin{equation*}
\begin{aligned}
\left|(F^{n})'(z)\right|&=\left|F'(F^{n-1}(z))\right|\cdot\left|(F^{n-1})'(z)\right|\\
&\geq \frac{\re F^{n}(z)-R}{4\pi}\cdot \prod_{j=0}^{n-2}\left|F'(F^{j}(z))\right|\\
&\geq \frac{\re F^{n}(z)-R}{4\pi}\cdot \left(\frac{R_1 -R}{4\pi}\right)^{n-1}.
\end{aligned}
\end{equation*}
Therefore, for $z\in T_n$ we have
\begin{equation*}
\begin{aligned}
r_{n}(z)=\frac{\re F^{n}(z)}{\left|(F^{n})'(z)\right|} &\leq \frac{\re F^{n}(z)}{\re F^n (z)-R}\cdot 4\pi \cdot \left(\frac{4\pi}{R_1 -R}\right)^{n-1}\\
&\leq \frac{R_1}{R_1 -R}\cdot 4\pi \cdot \left(\frac{4\pi}{R_1 -R}\right)^{n-1}\\
&= \left(\frac{4\pi}{R_1 -R}\right)^{n}R_1\\
&= \frac{R_1}{K^n}.
\end{aligned}
\end{equation*}
This implies (iii). Essentially (iv) follows from the above Lemma \ref{lemma 2}. Since $z\in L_n\subset T_n$, by definition of $T_n$ we have $\re z_k > R_1$ for $0\leq k\leq n$. Now since
\begin{equation*}
\frac{63}{64}\re z_n > \frac{63}{64} R_1 > \frac{63}{64}\cdot 2R >R,
\end{equation*}
we obtain that $Q(z_n)$ is contained in $H$. If $\phi$ is the inverse branch of $F$ which maps $z_n$ to $z_{n-1}$, then $\phi(Q(z_n))$ is contained in $W$ and hence the preimage of $Q(z_n)$ under the pullback of the inverse branch of $F^k$ which maps $z_n$ to $z_{n-k}$ is contained in $W$ for each $k=1,\dots, n$. If we denote by $\phi_n$ the branch of the inverse of $F^n$ which maps $z_n$ to $z$, then $\phi_n$ extends to a univalent map on $D(z_n, \frac{1}{2}\re z_n)$ since $\frac{1}{2}\re z_n > \frac{1}{2}R_1 > R$. By using Lemma \ref{lemma 2} and by taking $\sigma=1/256$ and $\tau=1/16$ we have
\begin{equation}\label{cb}
D\left(z, \sigma r_{n}(z)\right)\subset \phi_{n}\left(Q(z_n)\right)\subset D\left(z, \tau r_{n}(z)\right).
\end{equation}
The conclusion (iv) follows if we take 
\begin{equation*}
A_n(z)=\phi_{n}\left(Q(z_n)\right).
\end{equation*}
The last conclusion (v) follows if we consider the following square centred at $z_n$
\begin{equation*}
Q'(z_n)=\left\{z \in \mathbb{C}: \left|\re (z - z_n)\right|\leq \frac{1}{4} \re z_n, ~\left|\im (z - z_n) \right|\leq \frac{1}{4} \re z_n \right\}.
\end{equation*}
Similar arguments as above show that 
$$\phi_{n}\left(Q'(z_n)\right)\subset T_{n-1}$$
and
$$D\left(z, \tau r_{n}(z)\right)\subset\phi_{n}\left(Q'(z_n)\right).$$
Therefore, $D(z, \tau r_n (z))\subset T_{n-1}$ for $z\in L_n$. The conclusion (v) follows.\\

Now we split our proof into two steps. First we estimate the area of $T_{n-1}\setminus T_n$ in $D(z,\tau r_n(z))$ for $z\in T_n$, which we call \textit{local estimate}. Then we spread the local estimate to a \textit{global estimate}, which is the area of $T_{n-1}\setminus T_n$ in $P_n (w)$, by using the above (i), (ii) and (iii).\\

First we note that
\begin{equation*}
T_{n-1}\setminus T_n =F^{-n}(S)\cap T_{n-1}.
\end{equation*}
Together with (v) above we have
\begin{equation}\label{tt}
\begin{aligned}
&\area \left(\left(T_{n-1}\setminus T_n\right) \cap D(z,\tau r_{n}(z))\right)\\
&= \area \left(\left(F^{-n}(S)\cap T_{n-1}\right)\cap D(z,\tau r_{n}(z))\right)\\
&= \area \left(F^{-n}(S)\cap D(z,\tau r_{n}(z))\right).
\end{aligned}
\end{equation}
Recall our definition of $S$ in \eqref{sss} and $\theta(r)$ in \eqref{theta1}. We define
$$\varphi(x)=\meas\left\{~y\in [0,2\pi]: x+i y\in S~ \right\},$$
and
$$\varphi_{0}(x)=\theta_{0}(e^x).$$
Then $\varphi(x)=\theta(e^x)$. Since $\theta(x)\geq \theta_{0}(x)$ for large $x$, we see that $\varphi(x)\geq \varphi_{0}(x)$ for large $x$. Now we  can give a lower bound for the area of $S$ in the square $Q(z_n)$ given in (\ref{square}). For simplicity we put $Q=Q(z_n)$. We use the fact that $\theta_0$ is a continuous and decreasing function. Since the square $Q$ contains at least $[\frac{\re z_n}{64\pi}]$ horizontal strips of width $2\pi$ and $\re z_n > R_1 > 64\pi$, we obtain
\begin{equation}\label{areaaa}
\begin{aligned}
\area\left(S\cap Q\right) &\geq \left[\frac{\re z_n}{64\pi}\right]\int_{\frac{63}{64}\re z_n}^{\frac{65}{64}\re z_n}\varphi(t)dt\\
&\geq \left[\frac{\re z_n}{64\pi}\right]\int_{\frac{63}{64}\re z_n}^{\frac{65}{64}\re z_n}\varphi_{0}(t)dt\\
&\geq \left[\frac{\re z_n}{64\pi}\right]\varphi_{0} \left(\frac{65}{64}\re z_n\right) \frac{1}{32}\re z_n.
\end{aligned}
\end{equation}
Here $[~\cdot~]$ denotes the integer part. Therefore,
\begin{equation}\label{density-est}
\begin{aligned}
\dens\left(S, Q\right) &\geq \left[\frac{\re z_n}{64\pi}\right]  \dfrac{\varphi_{0}\left(\dfrac{65}{64}\re z_n\right) \dfrac{1}{32}\re z_n}{ \left(\dfrac{1}{32}\re z_n \right)^2} \\
&\geq \frac{1}{2}\frac{\re z_n}{64\pi} \dfrac{\varphi_{0}\left(\dfrac{65}{64}\re z_n\right)}{\dfrac{1}{32}\re z_n}\\
&=\frac{1}{4\pi}\varphi_{0}\left(\dfrac{65}{64}\re z_n\right)\\
&\geq \frac{1}{4\pi}\varphi_{0}\left(2\re z_n\right)\\
&=: \varphi_1 \left(\re z_n\right).
\end{aligned}
\end{equation}

As we mentioned above, $\phi_n$, which is the inverse branch of $F^n$ which maps $z_n$ to $z$, extends to a univalent map on $D(z_n, \frac{1}{2}\re z_n)$. Thus by Koebe's theorem, there exist positive constants $K_1$ and $K_2$ such that
\begin{equation*}\label{density estimate}
K_1 \dens\left(S,Q\right)\leq \dens\left(\phi_{n}(S), \phi_{n}(Q)\right)\leq K_2 \dens\left(S,Q\right)
\end{equation*}
for $n\in\mathbb{N}$. Then by using \eqref{cb} and \eqref{density-est} we have
\begin{equation*}
\begin{aligned}
\dens\left(F^{-n}(S), D(z,\tau r_{n}(z))\right)&=\dfrac{\area\left(F^{-n}(S)\cap D(z,\tau r_{n}(z))\right)}{\area D(z,\tau r_{n}(z))}\\
&\geq\dfrac{\area \left(\phi_{n}(S)\cap \phi_{n}(Q)\right)}{\area D(z,\tau r_{n}(z))}\\
&\geq\dfrac{K_1 \dens \left(S,Q)\cdot \area \phi_{n}(Q\right)}{\area D(z,\tau r_{n}(z))}\\
&\geq\dfrac{K_1 \dens \left(S,Q\right)\cdot \area D(z,\sigma r_{n}(z))}{\area D(z,\tau r_{n}(z))}\\
&= K_1 \left(\frac{\sigma}{\tau}\right)^2 \dens \left(S, Q\right)\\
&\geq K_1 \left(\frac{\sigma}{\tau}\right)^2 \varphi_{1}\left(\re z_n \right).
\end{aligned}
\end{equation*}
So by \eqref{tt}, for $z\in T_n$ we have, with $\varphi_2 (x)=K_{1}(\sigma/\tau)^2\varphi_1 (x)$,
\begin{equation}\label{locale}
\begin{aligned}
\area &\left(\left(T_{n-1}\setminus T_n\right) \cap D\left(z,\tau r_{n}(z)\right)\right)\\
&\geq K_1 \left(\frac{\sigma}{\tau}\right)^2 \varphi_{1}\left(\re z_n\right)\cdot \area D\left(z,\tau r_{n}(z)\right)\\
&= \varphi_2 (\re z_n)\cdot \area D(z,\tau r_{n}(z))).
\end{aligned}
\end{equation}

Since $f$ is of finite order, there exists some constant $\rho<\infty$ such that $\log\log M(r,f)\leq \rho\log r$. Now, we see that for any point $z$ with large real part,
\begin{equation}\label{compart}
\re F(z)\leq \exp(\rho\re z) =E(\rho\re z)\leq E^{2}(\re z),
\end{equation}
and hence
\begin{equation}\label{itee}
\re z_{k}=\re F^{k}(z)\leq E^{2k}(\re z).
\end{equation}

Now we can deduce our global estimate from the conclusions (i), (ii), (iii), \eqref{locale} and \eqref{itee} as follows:
\begin{equation}\label{globaless}
\begin{aligned}
&\area \left(\left(T_{n-1}\setminus T_n\right) \cap P_{n-1}\right)\\
&\geq \area \left(\left(T_{n-1}\setminus T_n\right) \cap \bigcup_{z\in L_{n}}D(z, \tau r_{n}(z))\right)\\
&=\sum_{z\in L_{n}}\area \left(\left(T_{n-1} \setminus T_{n}\right) \cap D(z, \tau r_{n}(z))\right)\\
&\geq \sum_{z\in L_{n}} \varphi_2 \left(\re z_n\right)\cdot \area D(z, \tau r_{n}(z))\\
&\geq \sum_{z\in L_{n}}\varphi_2 \left(E^{2n}(\re z)\right)\cdot \area D(z, \tau r_{n}(z))\\
&\geq \varphi_2 \left(E^{2n}\left(\frac{65}{64}\re w\right)\right) \sum_{z\in L_{n}} \area D(z, \tau r_{n}(z))\\
&\geq \frac{1}{16}\varphi_2 \left(E^{2n}\left(\frac{65}{64}\re w\right)\right)\cdot \area\left(\bigcup_{z\in L_{n}}D(z, 4\tau r_{n}(z))\right)\\
&\geq \frac{1}{16}\varphi_2 \left(E^{2n}\left(\frac{65}{64}\re w\right)\right)\cdot \area \left(T_{n}\cap P_n\right).
\end{aligned}
\end{equation}
Since
\begin{equation*}
\begin{aligned}
\area\left(\left(T_{n-1}\setminus T_n\right) \cap P_{n-1}\right)&=\area\left(T_{n-1}\cap P_{n-1}\right)-\area\left(T_{n}\cap P_{n-1}\right)\\
&\leq \area\left(T_{n-1}\cap P_{n-1}\right)-\area\left(T_{n}\cap P_{n}\right),
\end{aligned}
\end{equation*}
we obtain
$$\area\left(T_{n-1}\cap P_{n-1}\right)\geq \left[1+ \frac{1}{16}\varphi_2 \left(E^{2n}\left(\frac{65}{64}\re w\right)\right)\right]\cdot \area\left(T_n \cap P_n\right).$$
Let
\begin{equation*}
\begin{aligned}
P_{\infty}&=\bigcap_{n\geq 1}P_n\\
&=\left\{z:~|\re(z-w)|\leq \frac{\re w}{M}-c, |\im(z-w)|\leq \frac{\re w}{M}-c\right\},
\end{aligned}
\end{equation*}
where
$$c=\sum_{j=1}^{\infty}\frac{\tau R_1}{K^j}= \frac{\tau R_1}{K-1}.$$
We have
\begin{equation*}
\begin{aligned}
\area\left(T_n \cap P_{\infty}\right)&\leq \area\left(T_n \cap P_n\right)\\
&\leq \prod_{k=1}^{n}\frac{1}{1+ \dfrac{1}{16}\varphi_2 \left(E^{2k}\left(\frac{65}{64}\re w\right)\right)}\cdot \area\left(T_1 \cap P_1\right),
\end{aligned}
\end{equation*}
which means that
\begin{equation}\label{infinite product}
\area\left(T\cap P_{\infty}\right)\leq \prod_{k=1}^{\infty}\frac{1}{1+ \dfrac{1}{16}\varphi_2 \left(E^{2k}\left(\frac{65}{64}\re w\right)\right)}\cdot \area\left(T_1 \cap P_1\right).
\end{equation}
Since
\begin{equation*}
\begin{aligned}
\frac{1}{16}\varphi_2 \left(E^{2k}\left(\frac{65}{64}\re w\right)\right)&=\alpha \cdot \varphi_{0}\left(2E^{2k}\left(\dfrac{65}{64}\re w\right)\right)\\
&\geq \alpha \cdot \varphi_{0}\left(E^{2k+2}\left(\re w\right)\right)\\
&= \alpha \cdot \theta_{0}\left(E^{2k+4}\left(\re w\right)\right),
\end{aligned}
\end{equation*}
where $\alpha=\frac{K_1}{64\pi}(\sigma/\tau)^2$. Since
\begin{equation}\label{compart2}
\begin{aligned}
\sum_{k=1}^{\infty}\theta_{0}\left(E^{k}(x)\right)&=\sum_{k=1}^{\infty}\theta_0\left(E^{2k}(x)\right) + \sum_{k=1}^{\infty}\theta_0\left(E^{2k-1}(x)\right)\\
&\leq \sum_{k=1}^{\infty}\theta_{0}\left(E^{2k}(x)\right) + \sum_{k=1}^{\infty}\theta_0\left(E^{2k-2}(x)\right)\\
&= 2\sum_{k=1}^{\infty}\theta_{0}\left(E^{2k}(x)\right) + \theta_{0}(x),
\end{aligned}
\end{equation}
by condition (\ref{theta-condition}) we see that
$$\sum_{k=1}^{\infty}\frac{1}{16}\varphi_2 \left(E^{2k}\left(\frac{65}{64}\re w\right)\right) =\infty.$$
This implies that
$$ \prod_{k=1}^{\infty}\dfrac{1}{1+ \dfrac{1}{16}\varphi_2 \left(E^{2k}\left(\dfrac{65}{64}\re w\right)\right)} =0,$$
which, together with (\ref{infinite product}), finishes the proof:
$$\area(T\cap P_{\infty})=0.$$
Since the point $w\in W$ is chosen arbitrarily, we have in particular that $\area T =0$. By the discussion at the beginning of this section we finally have $\area\j(f)=0$ and in particular $\area\i(f)=0$.
\end{proof}

\bigskip
\begin{proof}[Proof of Theorem \ref{thm1.2}]
Let $f\in\b$ be as in Theorem \ref{thm1.2}. Now we consider
$$f_{\lambda}: \c\to\c;~~z\mapsto f(\lambda z).$$
By choosing $\lambda$ to be sufficiently small, the function $f_{\lambda}$ will satisfy all conditions in Theorem \ref{main theorem}. Therefore, we see that $\area\j(f_{\lambda})=0$. By definition, $f$ and $f_{\lambda}$ are equivalent near infinity. By Theorem \ref{rempe}, there exists $R>0$ and a quasiconformal mapping $\theta:\c\to\c$ such that
$$\theta\circ f_{\lambda} = f \circ \theta$$
on the set
$$\j_{R}(f_{\lambda}):=\left\{z\in\c: |f_{\lambda}^{n}(z)|\geq R~\text{~for all~}~n\geq 1\right\}.$$
Now we put
$$\i_{R}(f_{\lambda}):=\i(f_{\lambda})\cap\j_{R}(f_{\lambda}).$$
Then $\area\i_{R}(f_{\lambda})=0$. So we see that there exists a constant $R'>0$ such that
$$\i_{R'}(f)\subset\theta(\i_{R}(f_{\lambda})).$$
Recall that
$$\i(f)=\bigcup_{n\geq 0}f^{-n}\left(\i_{R'}(f)\right).$$
Therefore, it is clear that $\area \i(f)=0$.
\end{proof}

\smallskip
\begin{remark}
The hypothesis that $f$ has finite order was used only in \eqref{compart} to conclude that $\re F(z)\leq E^{2}(\re z)$ if $\re z$ is large enough. The proof goes through with only minor modifications if instead we only have $\re F(z)\leq E^{N}(\re z)$ for some $N\in\mathbb{N}$ and $\re z$ sufficiently large. This implies that the condition in Theorem \ref{thm1.2} and Theorem \ref{main theorem} that $f$ has finite order can be replaced by the condition that
\begin{equation*}\label{newin}
\log ^{N}M(r,f)\leq r
\end{equation*}
for some $N\in\mathbb{N}$ and large $r$.
\end{remark}

\section{Construction of an entire function: proof of Theorem \ref{main result}}

The aim of this section is to construct an entire function which satisfies our conditions in Theorem \ref{main result} and hence shows the sharpness of the Aspenberg-Bergweiler condition (\ref{abb}). Recall that in the introduction we consider $E_{\beta}(z)=e^{\beta z}$, where $\beta\in (0,1/e)$. We obtain a local holomorphic solution $\Phi$ of the corresponding Schr\"oder's functional equation around the repelling fixed point $\xi$. Now we define a function $\varepsilon: (\xi,\infty)\to (0,\infty)$ by
\begin{equation}\label{ano}
\varepsilon(x)=\frac{1}{\log \Phi(x)}.
\end{equation}
Then this function tends to zero slower than any of the functions $1/\log^m$ where $m\in\mathbb{N}$. Recall that $\log^j$ denote the $j$-the iterate of the logarithm. The following estimate will be useful.

\begin{lemma}\label{4.2}
For $\varepsilon(x)$ defined above and for $N\in\mathbb{N}_0$, we have the following estimate:
\begin{equation}\label{varepsil}
\varepsilon'(x)\prod_{j=0}^{N}\log^{j}{x} \leq \varepsilon(x)^3
\end{equation}
for large $x$.
\end{lemma}

\begin{proof}

We recall how the function $\Phi$ in \eqref{s11} is constructed; see, for instance, \cite[Section 8]{milnor10}. First let $T_{\xi}(x)=x+\xi$ and $L_{\beta}(x)=E_{\beta}^{-1}(x)$. Define
\begin{equation}\label{Lx}
L=T_{\xi}^{-1}\circ L_{\beta}\circ T_{\xi}.
\end{equation}
Thus,
$$L(x)=\frac{\log (x+\xi)}{\beta}-\xi.$$
The function $L$ satisfies that  $L(0)=0$, $L'(0)=1/\lambda$ (recall that $\lambda=\beta\xi$). Since $\lambda$ is greater than one, $0$ is an attracting fixed point of $L$. Schr\"oder's functional equation has a unique local holomorphic solution $\Psi(z)$ normalised by $\Psi(0)=0$ and $\Psi'(0)=1$, such that
\begin{equation}\label{s}
\Psi(L(x))=\frac{1}{\lambda}\Psi(x).
\end{equation}
Then $\Psi(x)=\Phi(x+\xi)$.
Define
$$\Psi_{n}(x)=\lambda^n L^{n}(x),$$
then $\Psi(x)=\lim_{n\to\infty}\Psi_{n}(x)$. We can now compute explicitly the first derivative of $\Psi_n$:
\begin{equation*}
\begin{aligned}
\Psi'_{n}(x)&= \lambda^{n}\left(L^n\right)'(x)=\lambda^{n}L'\left(L^{n-1}\right)\cdot\left(L^{n-1}\right)'(x)\\
&=\dfrac{\lambda^n}{\beta\left(L^{n-1}(x)+\xi\right)}\cdot\left(L^{n-1}\right)'(x)=\lambda^n\prod_{j=0}^{n-1}\dfrac{1}{\beta\left(L^{j}(x)+\xi\right)}\\
&=\left(\frac{\lambda}{\beta\xi}\right)^n \cdot \prod_{j=0}^{n-1}\dfrac{1}{1+L^{j}(x)/\xi}= \prod_{j=0}^{n-1}\dfrac{1}{1+L^{j}(x)/\xi}.
\end{aligned}
\end{equation*}
Now we see that, for any $N\in\mathbb{N}_0$,
\begin{equation*}
\begin{aligned}
\Psi'_{n}(x)\prod_{j=0}^{N}\log^{j} x &=\prod_{j=0}^{n-1}\dfrac{1}{1+L^{j}(x)/\xi} \prod_{j=0}^{N}\log^{j} x\\
&= \prod_{j=0}^{N} \frac{\log^{j} x}{1+L^{j}(x)/\xi} \prod_{j=N+1}^{n-1}\dfrac{1}{1+L^{j}(x)/\xi}.
\end{aligned}
\end{equation*}

For any fixed finite integer $N$ (to our following applications, $N\leq 5$ is enough), the definition of $L(x)$ in (\ref{Lx}) can be put into the above equality. An easy computation shows that
\begin{equation}\label{phia}
\Psi'_{n}(x)\prod_{j=0}^{N}\log^{j} x \leq \frac{1}{2},
\end{equation}
for sufficiently large $x$. If we define 
\begin{equation}\label{defin}
\begin{aligned}
\eta (x)&=\frac{1}{\log\Psi(x)},\\
\eta_n (x)&=\frac{1}{\log\Psi_{n}(x)},
\end{aligned}
\end{equation}
then $\eta(x)=\lim_{n\to\infty}\eta_{n}(x)$ and also $\eta'(x)=\lim_{n\to\infty}\eta'_{n}(x)$. Then, by applying (\ref{phia}) and the definition of $\eta_{n}(x)$ in (\ref{defin}) we have
\begin{equation*}
\begin{aligned}
\eta_{n}'(x)\prod_{j=0}^{N}\log^{j}{x} &=\frac{1}{\left[\log \Psi_{n}(x)\right]^2}\frac{1}{\Psi_{n}(x)}\Psi'_{n}(x)\prod_{j=0}^{N}\log^{j}{x}\leq \frac{1}{2\Psi_{n}(x)\left[\log \Psi_{n}(x)\right]^2}\\
& \leq \frac{1}{2\left[\log \Psi_{n}(x)\right]^3} = \frac{1}{2}{\eta_{n}(x)^3}.
\end{aligned}
\end{equation*}

Now we see that
\begin{equation*}
\eta'(x)\prod_{j=0}^{N}\log^{j}{x} \leq \frac{1}{2}\eta(x)^3
\end{equation*}
for large $x$. Recall the definition of $\varepsilon(x)$ and $\Phi(x)$. We thus have $\eta(x)=\varepsilon(x+\xi)$. For large $x\in(\xi,\infty)$, we see that
\begin{equation*}
\begin{aligned}
\varepsilon'(x+\xi)\prod_{j=0}^{N}\log^{j}(x+\xi) &= \left(\eta'(x)\prod_{j=0}^{N}\log^{j}{x}\right) \frac{\prod_{j=0}^{N}\log^{j}(x+\xi)}{\prod_{j=0}^{N}\log^{j}{x}}\\
&\leq \frac{1}{2}\eta(x)^3  \frac{\prod_{j=0}^{N}\log^{j}(x+\xi)}{\prod_{j=0}^{N}\log^{j}{x}}\\
&= \frac{1}{2}\varepsilon(x+\xi)^3  \frac{\prod_{j=0}^{N}\log^{j}(x+\xi)}{\prod_{j=0}^{N}\log^{j}{x}}\\
&\leq \varepsilon(x+\xi)^3.
\end{aligned}
\end{equation*}
This finishes our proof.
\end{proof}

A function $\rho(r)$ defined on $[r_0,\infty)$, where $r_0>0$, is called a \textit{proximate order} if it satisfies the following conditions:
\begin{itemize}
\item[(1)] $\rho(r)\geq 0$;
\item[(2)] $\lim_{r\to\infty}\rho(r)=\rho$;
\item[(3)] $\rho(r)$ is continuously differentiable on $[r_0,\infty)$;
\item[(4)] $\lim_{r\to\infty} r\rho'(r)\log r =0$.
\end{itemize}
See \cite{goldbergmero} for a complete discussion of proximate orders. Recall that $\varepsilon(r)$ is defined in \eqref{ano}. Put
\begin{equation}\label{defrho}
\rho(r)=\frac{1}{2}+\varepsilon(r).
\end{equation}
Now we prove the following lemma.

\begin{lemma}\label{prox}
$\rho(r)$ is a proximate order.
\end{lemma}

\begin{proof}
To show the statement is true we only need to check whether $r\rho'(r)\log r\to 0$ as $r\to \infty$. This follows easily from the above lemma. Taking $N=1$ in (\ref{varepsil}), we see that
\begin{equation*}
\begin{aligned}
\left| r\rho'(r)\log r \right| &= \left| \varepsilon'(r) r \log r \right| \leq A[\varepsilon(r)]^3 \to 0 ~\text{~~as~~}~ r\to\infty.
\end{aligned}
\end{equation*}
The other conditions are easy to verify. We omit it here.
\end{proof}

We begin our construction of an entire function. Let $\varepsilon(r)$, $\rho(r)$ be as above. Let $\{a_n\}_{n\geq 0}$ be a positive real sequence tending to infinity which is chosen such that
\begin{equation*}
1\leq a_0\leq a_1 \leq \dots ,
\end{equation*}
and
\begin{equation}\label{nr}
n(r)=r^{\rho(r)}+\mathcal{O}\left(1\right),
\end{equation}
where $n(r)$ counts the number of elements $a_n$ which satisfies $a_n \leq r$. Recall that the exponent of convergence of the sequence $\{a_n\}_{n\geq 1}$ is defined by
\begin{equation*}
\lambda:=\inf\left\{\mu>0:~\sum_{k=0}^{\infty}\frac{1}{a_{k}^{\mu}}<\infty\right\}.
\end{equation*}
Moreover, it is well known that
\begin{equation*}
\lambda=\limsup_{r\to\infty}\frac{\log n(r)}{\log r}.
\end{equation*}
Then it is easy to see that $\lambda=1/2$. Therefore, the infinite product
\begin{equation}\label{f}
f(z)=\prod_{n=0}^{\infty}\left(1-\frac{z}{a_n}\right)
\end{equation}
converges locally uniformly and hence it is an entire function. Note that $\varepsilon(r)^{3} r^{\rho(r)}\to\infty$ as $r\to \infty$. Instead of \eqref{nr} it suffices to assume that
\begin{equation}\label{n}
n(r)=r^{\rho(r)}+\mathcal{O}\left(\varepsilon(r)^{3} r^{\rho(r)}\right),
\end{equation}
and the infinite product defined in above way is still an entire function. Now we consider the asymptotic behaviour of $f$ outside a small unbounded domain containing the positive real axis. We prove the following important property for the above function $f$. For a related result, we refer to \cite[Theorem 1.5]{bergweiler17}.

\begin{lemma}[Asymptotic representation]\label{asymp}
For $f$, $\varepsilon$ and $\rho$ defined above, we have
\begin{equation}\label{asmptot}
\log|f(re^{i\theta})|=\frac{\pi\cos((\theta-\pi)\rho(r))}{\sin(\pi\rho(r))}r^{\rho(r)}+\mathcal{O}\left(\varepsilon(r)^2 r^{\rho(r)}\right), 
\end{equation}
for
\begin{equation}\label{theta}
\varepsilon(r)\leq \theta \leq 2\pi-\varepsilon(r)
\end{equation}
as $r\to\infty$.
\end{lemma}

\begin{proof}
Following the standard argument using Riemann-Stieltjes integral (see \cite[Chapter 2, Section 5]{goldbergmero}), we have
$$\log f(z)=-z\int_{0}^{\infty} \frac{n(t,0)}{t(t-z)}dt$$
and
\begin{equation}\label{I}
I(z):=\int_{0}^{\infty}\frac{t^{\rho(r)}}{t(t-z)}dt=-\dfrac{\pi e^{-i\pi\rho(r)}}{\sin (\pi\rho(r))}z^{\rho(r)-1},
\end{equation}
defined for $0<\arg z<2\pi$. Here $n(t,0)$ denotes the number of zeros of $f$ in the disk $\{|z|<t\}$ and is equivalent to \eqref{nr}. For $z=re^{i\theta}$,
\begin{equation}\label{zI}
\re\left(zI(z)\right)=-\dfrac{\pi \cos(\rho(r)(\pi-\theta))}{\sin(\pi(\rho(r))}r^{\rho(r)}.
\end{equation}
By (\ref{n}), (\ref{I}) and (\ref{zI}) we have
\begin{equation}\label{int}
\begin{aligned}
\left|\log|f(z)|+\re\left(zI\left(z\right)\right)\right|&=\left|\re\left( z\int_{0}^{\infty}\dfrac{t^{\rho(r)}-n(t,0)}{t(t-z)}dt \right)\right|\\
&\leq r\int_{0}^{\infty}\dfrac{\left|n(t,0)-t^{\rho(r)}\right|}{t\left|t-z\right|}dt.
\end{aligned}
\end{equation}
To estimate the integral on the right-hand side of (\ref{int}), we consider
\begin{equation}\label{abde}
b(r)=\frac{1}{a(r)}=\exp\left\{\varepsilon(r)^3 \log\log r \right\}.
\end{equation}
We claim two properties of $a(r)$ and $b(r)$. First, $b(r)\to\infty$ as $r\to\infty$. This follows easily from (\ref{ano}) and (\ref{phi}), from which we see that $\Phi(r)\leq \log^{k}(r)$ for large $r$ and for any $k\in\mathbb{N}$. Hence, $\varepsilon(r)\geq {1}/{\log^{k}(r)}$ for any $k\geq 1$. Now the first claim follows. The second one is that, for any positive constant $\delta<1$,
\begin{equation}\label{ab}
a(r)^{\delta}=\frac{1}{b(r)^{\delta}}=o\left(\varepsilon(r)^2\right).
\end{equation}
In fact it can easily be seen that, this is true if the following holds:
\begin{equation*}
\frac{\varepsilon^{3}(r)\log\log r}{\log \varepsilon(r)}\to-\infty.
\end{equation*}
But this follows since $1/\log\varepsilon(r)\leq -\varepsilon(r)$ for large $r$ and $\varepsilon(r)^4\log\log r\to\infty$ as $r\to\infty$ since $\varepsilon(r)$ tends to $0$ slower than $1/\log^m r$ for every $m\in\mathbb{N}$.\\

Now following the argument in the proof of Theorem 2.2 in \cite[Chapter 2, Section 2]{goldbergmero}, for every $k\in [a(r), b(r)]$ there exists $k^* \in [a(r), b(r)]$ such that
\begin{equation}\label{rhominus}
\begin{aligned}
\left|\rho(kr)-\rho(r)\right|&=k^{*}r \left|\rho'(k^{*}r)\right|\left|\log k\right|\\
&= k^{*}r\left|\rho'(k^{*}r)\right|\left|\log (k^{*}r)\right|\frac{\left|\log k\right|}{\left|\log (k^{*}r)\right|}.
\end{aligned}
\end{equation}
By \hyperref[prox]{Lemma \ref{prox}}, $\rho(r)$ is a proximate order. The definition of proximate orders and \hyperref[4.2]{Lemma \ref{4.2}} then imply that
\begin{equation}\label{erm1}
k^{*}r\left|\rho'(k^{*}r)\right|\left|\log (k^{*}r)\right| =k^{*}r\left|\varepsilon'(k^{*}r)\right|\left|\log (k^{*}r)\right| \leq \frac{1}{\log\log (k^{*} r)}
\end{equation}
for large $r$. Recall that $a(r)$ and $b(r)$ are defined in (\ref{abde}). Since $k\leq b(r)$ and $k^{*}\geq a(r)$, we see that
\begin{equation}
\frac{\left|\log k\right|}{\left|\log (k^{*}r)\right|}\leq\frac{\log b(r)}{\log (r a(r))}=\frac{\log b(r)}{\log r - \log b(r)}=\frac{\varepsilon(r)^{3} \log\log r}{\log r-\varepsilon(r)^3 \log\log r},
\end{equation}
which, by a simple computation, \eqref{rhominus} and (\ref{erm1}), yields
$$\left|\rho(kr)-\rho(r)\right|\leq\left(1+o\left(1\right)\right)\frac{\varepsilon(r)^3}{\log r}  ~\text{~as~}~r\to\infty.$$


Therefore we see that, for all such $k$,
\begin{equation}\label{keyin}
\begin{aligned}
r^{\rho(kr)-\rho(r)}&=\exp\left\{(\rho(kr)-\rho(r))\log r\right\}\\
&=\exp\left\{(1+o(1)) \varepsilon(r)^3\right\}~\text{as~}~r\to\infty.
\end{aligned}
\end{equation}

Take $\delta=1/4$, then $\delta < \min\{\rho(t), 1-\rho(t)\}$ for large $t$ since $\rho(t)\to 1/2$ by definition. Moreover, note that for such $\delta$, we have the estimate \eqref{ab}. Now we separate the integral on the right-hand side of (\ref{int}) into three parts as follows:
\begin{equation*}
\begin{aligned}
\int_{0}^{\infty}\dfrac{\left|n(t,0)-t^{\rho(r)}\right|}{t\left|t-z\right|}dt&= \int_{0}^{a(r)r}\dfrac{\left|n(t,0)-t^{\rho(r)}\right|}{t\left|t-z\right|}dt+\int_{b(r)r}^{\infty}\dfrac{\left|n(t,0)-t^{\rho(r)}\right|}{t\left|t-z\right|}dt\\
&+ \int_{a(r)r}^{b(r)r}\dfrac{\left|n(t,0)-t^{\rho(r)}\right|}{t\left|t-z\right|}dt .
\end{aligned}
\end{equation*}

For the first integral, by using standard properties of the proximate order, (\ref{ab}), (\ref{keyin}), and the fact that for $t\in[0, a(r)r]$, we have $|t-z|\geq r(1-a(r))\geq r/2$, we find that there exists a constant $C$ such that
\begin{equation}\label{firste}
\begin{aligned}
\int_{0}^{a(r)r}\dfrac{\left|n(t,0)-t^{\rho(r)}\right|}{t\left|t-z\right|}dt &\leq \frac{C}{r}\left\{\int_{1}^{a(r)r}t^{\rho(t)-1}dt +\int_{1}^{a(r)r}t^{\rho(r)-1}dt  \right\}\\
&\leq C a(r)^{\delta} r^{\rho(r)-1}\\
&= o\left(\varepsilon(r)^2 r^{\rho(r)-1}\right).
\end{aligned}
\end{equation}
For the second integral, $|t-z|\geq t/2$ for $t\geq b(r)r$ and the same argument as above result in the following estimates:
\begin{equation}\label{seconde}
\begin{aligned}
\int_{b(r)r}^{\infty}\dfrac{\left|n(t,0)-t^{\rho(r)}\right|}{t\left|t-z\right|}dt &\leq C \left\{\int_{b(r)r}^{\infty} t^{\rho(t)-2}dt +\int_{b(r)r}^{\infty}t^{\rho(r)-2}dt  \right\}\\
&\leq C \frac{1}{b(r)^{\delta}} r^{\rho(r)-1}\\
&= C a(r)^{\delta} r^{\rho(r)-1}\\
&= o\left(\varepsilon(r)^2 r^{\rho(r)-1}\right).
\end{aligned}
\end{equation}
For $t\in[a(r)r, b(r)r]$ and $z=re^{i\theta}$, we have $|t-z|\geq (t+r)\sin\frac{\varepsilon(r)}{2}$. Combining this with (\ref{n}) and (\ref{keyin}), we can obtain an estimate of the last integral as follows:
\begin{equation}\label{thirde}
\begin{aligned}
\int_{a(r)r}^{b(r)r}\dfrac{\left|n(t,0)-t^{\rho(r)}\right|}{t\left|t-z\right|}dt&\leq\int_{a(r)r}^{b(r)r}\dfrac{\left|n(t,0)-t^{\rho(t)}\right|}{t\left|t-z\right|}dt\\
&+\int_{a(r)r}^{b(r)r}\dfrac{\left|t^{\rho(t)}-t^{\rho(r)}\right|}{t\left|t-z\right|}dt\\
&\leq\frac{C}{\sin\frac{\varepsilon\left(r\right)}{2}}\int_{a(r)r}^{b(r)r}\frac{\varepsilon\left(t\right)^3 t^{\rho(t)}}{t\left(t+r\right)}dt\\
&=\mathcal{O}\left(\frac{\varepsilon\left(r\right)^2}{r}\int_{a(r)}^{b(r)}\frac{(\tau r)^{\rho(\tau r)}}{\tau\left(1+\tau\right)}d\tau   \right)\\
&= \mathcal{O}\left(\varepsilon\left(r\right)^2 r^{\rho(r)-1}\int_{0}^{\infty}\frac{\tau^{\rho(r)}}{\tau\left(1+\tau\right)}d\tau   \right)\\
&= \mathcal{O}\left(\varepsilon\left(r\right)^2 r^{\rho(r)-1}  \right).
\end{aligned}
\end{equation}

Putting all the estimates (\ref{int}), (\ref{firste}), (\ref{seconde}) and (\ref{thirde}) together, we see that for $z=re^{i\theta}$ with $\theta$ satisfying (\ref{theta}), $\log\left|f(z)\right|$ has the following asymptotic representation:
$$\left|\log|f(z)|+\re(zI(z))\right|=\mathcal{O}\left(\varepsilon(r)^2 r^{\rho(r)}\right).$$
\end{proof}

In the following, for some $r_0 > 0$ we define
$$\gamma^{+}=\left\{re^{i\varepsilon(r)}: r\geq r_0\right\},$$
$$\gamma^{-}=\left\{re^{-i\varepsilon(r)}: r\geq r_0\right\},$$
and
$$G(\gamma)=\mathbb{C}\setminus\left\{re^{i\theta}\in\mathbb{C}: \varepsilon(r)\leq \theta \leq 2\pi-\varepsilon(r),~r\geq r_0\right\}.$$

A consequence of the above lemma is the following fact.

\begin{lemma}\label{boundness}
The function $f$ is bounded on $\gamma^{+}$ and $\gamma^{-}$.
\end{lemma}

\begin{proof}
This follows from the asymptotic representation of $f$ given in \hyperref[asymp]{Lemma \ref{asymp}}. For sufficiently large $r$,
$$\rho(r)(\pi-\varepsilon(r))=\frac{\pi}{2}+\varepsilon(r)\left(\pi-\frac{1}{2}-\varepsilon(r)\right)>\frac{\pi}{2}.$$
So we have
\begin{equation*}
\begin{aligned}
\cos\left(\rho(r)(\pi-\varepsilon(r)\right)&=-\sin\left(\varepsilon(r)\left(\pi-\frac{1}{2}-\varepsilon(r)\right)\right)\\
&=-(1+o(1))\left(\pi-\frac{1}{2}\right)\varepsilon(r).
\end{aligned}
\end{equation*}
Since $\varepsilon(r)^2 =o\left(\varepsilon(r)\right)$, it follows from the asymptotic representation of $\log|f|$ in \hyperref[asymp]{Lemma \ref{asymp}} that
$$\log\left|f(re^{i\theta})\right|<0,$$
for $|\theta|=\varepsilon(r)$ and for $r$ sufficiently large.
\end{proof}

\begin{lemma}\label{bdd}
$f(z)$ is bounded in $G(\gamma)$.
\end{lemma}

\begin{proof}
By Denjoy-Carleman-Ahlfors-Theorem, if $g$ is entire then the number of components of $\left\{z\in\mathbb{C}: |g(z)|>R\right\}$ for $R>0$ is less than or equal to $\max\left\{1, 2\rho(g)\right\}$, where $\rho(g)$ is the order of $g$. Therefore, our function $f$, which is constructed in (\ref{f}) and has order of growth $1/2$, has at most one tract. However, the asymptotic formula (\ref{asmptot}) for $f$ in \hyperref[asymp]{Lemma \ref{asymp}} implies that $f(z)$ is unbounded when $z$ goes to infinity along the negative real axis. So this means that $f$ has a tract containing the negative real axis. By \hyperref[boundness]{Lemma \ref{boundness}} $f$ is bounded on $\gamma^+$ and $\gamma^-$. Therefore, the only tract of $f$ should be contained in $\mathbb{C}\setminus G(\gamma)$, which means that $f$ is bounded in $G(\gamma)$.
\end{proof}

\begin{lemma}\label{B}
The function $f$ constructed in \eqref{f} is in the Eremenko-Lyubich class $\b$.
\end{lemma}

\begin{proof}
We have
$$f(z)=\lim_{n\to\infty} P_{n}(z),$$
where
$$P_{n}(z)=\prod_{k=1}^{n}\left(1-\frac{z}{a_k}\right).$$
Thus also $P'_{n}\to f'$. To show the entire function $f$ we constructed in  (\ref{f}) is in the Eremenko-Lyubich class $\b$, we need to prove that the set of all critical values and asymptotic values of $f$ is bounded. Since all the polynomials $P'_{n}$ have only real positive zeros, Hurwitz's theorem implies that all zeros of $f'$ are also real and positive. So $f$ has only real positive critical points. As $[r_0, \infty)$ lies in the unbounded domain $G(\gamma)$ defined above, \hyperref[bdd]{Lemma \ref{bdd}} means that the set of critical values of $f$ is bounded. Moreover, $f$ has at most one asymptotic value by the Denjoy-Carleman-Ahlfors theorem. Therefore, $f$ belongs to the Eremenko-Lyubich class $\b$.
\end{proof}

\begin{remark}
The \textit{Laguerre-P\'olya class} $\mathcal{LP}$ consists of entire functions which are locally uniform limits of real polynomials with real zeros. Since the $P_n$ have real zeros, $f\in\mathcal{LP}$. The argument used in the proof of Lemma \ref{B} shows that whenever $f\in\mathcal{LP}$, then $f'$ has only real zeros. In fact, if $f\in\mathcal{LP}$, then $f'\in\mathcal{LP}$.
\end{remark}

To show that the escaping set of the function we constructed above has zero Lebesgue measure, we need to apply \hyperref[main theorem]{Theorem \ref{main theorem}}. To do this, we need the following lemma.

\begin{lemma}\label{infinite}
For $\varepsilon(x)={1}/{(\log \Phi(x))}$ and $x>0$, we have
$$\sum_{k=1}^{\infty}\varepsilon\left(E^{k}\left(x\right)\right)=\infty.$$
\end{lemma}

\begin{proof}
Recall that $E_{\beta}(x)=e^{\beta x}$, $E(x)=e^x$. We shall show first that, for any given $x_1$ there exists $x_2>x_1$ such that for any $k\in\mathbb{N}$,
$$E_{\beta}^{k}\left(x_2\right)\geq E^{k}\left(x_1\right).$$
To see this we first consider the following real function defined as
$$F_{\beta}(x)=\beta e^x .$$
In particular, $F_{1}(x)=E(x)$. An easy computation shows that $F_{\beta}(\beta x)=\beta E_{\beta}(x)$. Let $c>\log(2/\beta)$ and $e^x > c$. We have
$$F_{\beta}(x+c)=\beta e^c e^x>2e^x>e^x +c=F_{1}(x)+c>F_{1}(x).$$
Since for any $k\in\mathbb{N}$, $E^{k}(x) >c$, by induction we have
\begin{equation}\label{neq}
F_{\beta}^{k}\left(x+c\right)\geq F_{1}^{k}\left(x\right)+c>F_{1}^{k}\left(x\right).
\end{equation}
If we take $x_1 =x$ and $x_2 ={(x+c)}/{\beta}$, then
$$F_{\beta}^{k}\left(x+c\right)=F_{\beta}^{k}\left(\beta x_2 \right)=\beta E_{\beta}^{k}\left(x_2\right).$$
Together with (\ref{neq}) we see that
$$E_{\beta}^{k}\left(x_2\right)>\beta E_{\beta}^{k}\left(x_2\right)=F_{\beta}^{k}\left(x+c\right)>F_{1}^{k}\left(x_1\right)=E^{k}\left(x_1\right).$$
Our assertion now follows. So for $x>0$ there exists $x'>x$ with $E_{\beta}^{k}(x')>E^{k}(x)$. Thus
\begin{equation*}
\begin{aligned}
\sum_{k=1}^{\infty}\varepsilon(E^{k}(x))&=\sum_{k=1}^{\infty}\frac{1}{\log \Phi\left(E^{k}\left(x\right)\right)}\\
&\geq \sum_{k=1}^{\infty}\frac{1}{\log \Phi\left(E_{\beta}^{k}\left(x'\right)\right)}\\
&=\sum_{k=1}^{\infty}\frac{1}{\log\left(\lambda^k \Phi\left(x'\right)\right)}\\
&= \sum_{k=1}^{\infty}\frac{1}{k\log \lambda + \log \Phi(x')}\\
&=\infty.
\end{aligned}
\end{equation*}
\end{proof}

\begin{proof}[Proof of Theorem \ref{main result}]
We consider the transcendental entire function $f$ constructed in (\ref{f}) which satisfies (\ref{n}). Lemma \ref{bdd} and Lemma \ref{B} together imply that $f$ belongs to the class $\b$ and is bounded in $G(\gamma)$.
Since
$$\varepsilon(r)=\frac{1}{\log \Phi(r)},$$
to show that the Lebesgue measure of $\i(f)$ of $f$ is zero, we shall apply our Theorem \ref{main theorem}, in which we have $\theta_{0}(r)=2\varepsilon(r)$. Therefore,
\begin{equation}\label{diver}
2\sum_{k=1}^{\infty}\varepsilon\left(E^{k}\left(x\right)\right)=2\sum_{k=1}^{\infty}\frac{1}{\log \Phi\left(E^{k}\left(x\right)\right)}=\infty.
\end{equation}
Now Lemma \ref{infinite} implies that (\ref{diver}) is valid. Now, we need to check that the condition \eqref{examplec} is satisfied. This follows from our Lemma \ref{asymp}, which yields
\begin{equation*}
\begin{aligned}
\log M(r,f)&\leq \frac{\pi}{\sin(\pi\rho(r))}r^{\rho(r)}+\mathcal{O}\left(\varepsilon(r)^2 r^{\rho(r)}\right)\\
&=\left(1+o(1)\right)\pi r^{\rho(r)}.\\
\end{aligned}
\end{equation*}
Thus, we obtain that
\begin{equation*}
\begin{aligned}
\log\log M(r,f)&\leq \log r^{\rho(r)}+\mathcal{O}\left(1\right)\\
&=\rho(r)\log r + \mathcal{O}(1)\\
&= \left( \frac{1}{2}+\varepsilon(r) \right)\log r + \mathcal{O}(1).
\end{aligned}
\end{equation*}
By considering $\lambda f(z)$ for a sufficiently small positive constant $\lambda$ if necessary, the conditions in Theorem \ref{main theorem} are satisfied, which means that $\area\j(f)=0$ and in particular $\area\i(f)=0$ .
\end{proof}

\begin{remark}
The function we constructed above shows the sharpness of the Aspen\-berg-Bergweiler condition (\ref{abb}) in the case of only one logarithmic tract over infinity. To obtain an entire function with any finite number, say $N$, logarithmic tracts over infinity, we let $f$ be as in Theorem \ref{main result} and let $g(z)=f(z)^N$ and $h(z)=f(z^N)$. Then $h$ is in class $\b$ and has $N$ logarithmic tracts over infinity. Moreover, the function $h$ satisfies that
$$\log\log M(r,h)\leq \left(\frac{N}{2}+\frac{N}{\log \Phi(r)}\right)\log r +\mathcal{O}(1).$$
\end{remark}
Now we need to show that for such function $h$ we have $\area \i(h)=0$. This follows since $z\in \i(h)$ if and only if $z^N \in \i(g)$. Thus $\area\i(h)=0$ if and only if $\area\i(g)=0$. Now $\area\i(g)=0$ by the same argument that gives $\area\i(f)=0$. Hence $\area\i(h)=0$.

\begin{proof}[Proof of Theorem \ref{corothm1}]
The construction of an entire function in class $\b$ satisfying conditions in the theorem follows the line of proof of Theorem \ref{main result}. The key ingredient is to check that $f$ satisfies the asymptotic representation in \eqref{asmptot} with $\rho(r)$ as in \eqref{defrho}. This in fact follows from the properties of $\varepsilon(r)$ given in the theorem. We omit the details here.
\end{proof}

\section{Adaption to infinite-order entire functions}

We shall use some definitions and some basic facts from value distribution theory of meromorphic functions. See, for example, \cite{goldbergmero} and \cite{hayman1} for an introduction to the theory. We denote by $T(r,f)$ the \textit{Nevanlinna characteristic} of a meromorphic function $f$, and by $m(r,f)$ the \textit{proximity function} of $f$. The order of growth of a meromorphic function $f$ is defined as follows:
\begin{equation}\label{order2}
\rho(f)=\limsup_{r\to\infty}\frac{\log T(r,f)}{\log r}.
\end{equation}
Compare (\ref{order}) and (\ref{order2}). When $f$ is an entire function, $T(r,f)$ can be replaced by the term $\log M(r,f)$. A more precise relation between $T(r,f)$ and $\log M(r,f)$ when $f$ is entire is given by the following result. Moreover, we also need a fundamental result due to Nevanlinna. Note that the results below can be found in \cite[Chapter 1]{goldbergmero}. 
\begin{lemma}\label{maxi}
Let $f(z)$ be an entire function, and let $0<r<R$. Then
\begin{equation}\label{maxii}
T(r,f)\leq \log M(r,f) \leq \frac{R+r}{R-r}T(R,f).
\end{equation}
\end{lemma}

\begin{lemma}[First fundamental theorem of Nevanlinna theory]\label{nevan}
For $a\in\c$,
\begin{equation}
T(r,f)=T \left(r, \frac{1}{f-a}\right)+\mathcal{O}(1).
\end{equation}
\end{lemma}

We note that the second inequality in (\ref{maxii}) is even true when $M(r,f)$ is replaced by $M_{G}(r,f)$ for a meromorphic function $f$ with a logarithmic tract $G$ over $\infty$. Here $M_{G}(r,f)=\max_{z\in G, |z|=r}|f(z)|$. For convenience we state it as follows and refer to \cite[Chapter XI, Section 4.3]{nevanlinna} for a proof.

\begin{lemma}\label{maxag}
Let $f$ be a meromorphic function with a logarithmic tract $G$ over infinity and let $0<r<R$. Then
\begin{equation}
\log M_{G}(r,f) \leq \frac{R+r}{R-r}\left(T(R,f) + \mathcal{O}(1)\right).
\end{equation}
\end{lemma}

The dynamics of meromorphic functions with direct or logarithmic singularities share many properties with those of entire functions. We refer to \cite{bergweiler3} for detailed discussions of dynamics of these functions.

We recall the following result due to Tsuji \cite{tsuji1975potential}, which is connected to the proof of Denjoy-Carleman-Ahlfors theorem. See, for instance, \cite[Chapter $5$]{goldbergmero}. To formulate it precisely, we need some notations. For an unbounded domain $G$ with boundary $\Gamma$ and $r>0$ such that $\{z: |z|=r\}\cap \Gamma \neq \emptyset$, denote by 
$$\beta(r)=\meas\left\{\theta\in[0, 2\pi]: re^{i\theta}\in G\right\}.$$
If $\{z: |z|=r\}\cap \Gamma = \emptyset$ then we define $\beta(r)=\infty$.

\begin{lemma}\label{tsuji}
Let $G$ be an unbounded domain and $\Gamma$ its boundary. Let $f$ be continuous in $G\cup\Gamma$ and holomorphic in $G$.
Suppose that $f$ is bounded on $\Gamma$ but unbounded in $G$. Then
$$\log\log M_{G}(r,f)\geq \pi \int_{r_1}^{\alpha r}\frac{dt}{t\beta(t)}+ \mathcal{O}(1),$$
where $r_1 >0$ and $0<\alpha<1$ do not depend on $r$.
\end{lemma}

In our case, we can take the unbounded domain $G$ to be a logarithmic tract over $\infty$ and $\beta(r)$ the corresponding measure in this tract. Furthermore, by choosing $r_1>0$ large enough such that $\{z: |z|=r\}\cap \Gamma \neq \emptyset$ we have $\beta(r)>0$ for $r\geq r_1$.

\smallskip
As mentioned in the introduction, the proof of Theorem \ref{infinite-order2} is only a slightly modified version of that of Theorem \ref{main theorem} if we require our entire function $f$ to be of disjoint type, satisfying conditions in Theorem \ref{main theorem}. To transfer to entire functions not being of disjoint type,  we use a result of Rempe \cite[Theorem 1.1]{rempe8} as in the proof of Theorem \ref{thm1.2}.

\begin{proof}[Proof of Theorem \ref{infinite-order2}]

For the function $f$ given in the theorem, we define
\begin{equation*}
g(z):=\lambda f(z),
\end{equation*}
where $\lambda>0$ is a constant to be determined later. Then $g\in\b$. So we can apply a logarithmic change of variables to $g$ and hence we obtain a function $G$, corresponding to $g$, such that $G$ is a biholomophic map from every component of $W$ onto $H'$, where
$$H'=\left\{~z\in\c:~ \re z>\log r_0 + \log \lambda~\right\}.$$
Now the constant $\lambda$ is chosen sufficiently small such that
\begin{equation*}
\inf\{\re z:~z\in W\}>2(\log r_0 + \log \lambda)
\end{equation*}
and
\begin{equation*}
\inf\{\re z:~z\in W\}-(\log r_0 + \log \lambda) > 64\pi.
\end{equation*}

Therefore, for the function $g$ defined above satisfies the conditions in Theorem \ref{main theorem} by taking
\begin{equation*}
R_1:=\inf\{~\re z:~z\in W~\}.
\end{equation*}

Now the proof of the theorem is similar to that of Theorem \ref{main theorem}. We choose a point $w$ with large real part and a sequence of squares $P_{n}(w)$ exactly the same as those in \eqref{sequs}. In the same way we define the sets $T_n$ and $S$. Now there exists a countable subset $L_n$ of $T_n\cap P_n$ which satisfies conditions (i)-(v) as before. Now the proof follows if we can obtain appropriate estimates: the local estimate and the global estimate.

The local estimate is achieved by using the inequality (\ref{didi}). In terms of $\varphi$, where $\varphi(x)=\theta(e^x)$, the condition (\ref{didi}) can be written as
\begin{equation}\label{onekey}
\int_{cx}^{x}\varphi(s)ds=\int_{e^{cx}}^{e^x}\varphi(\log t)\frac{dt}{t}=\int_{e^{cx}}^{e^x}\theta(t)\frac{dt}{t}\geq \frac{x}{A(e^x)}=\frac{x}{A(E(x))}.
\end{equation}
Then since $63/65\leq c <1$, by letting $z_n:=G^{n}(z)$ we obtain, instead of \eqref{areaaa},
\begin{equation*}
\begin{aligned}
\area\left(S\cap Q\right) &\geq \left[\frac{\re z_n}{64\pi}\right]\int_{\frac{63}{64}\re z_n}^{\frac{65}{64}\re z_n}\varphi(t)dt\\
&\geq \left[\frac{\re z_n}{64\pi}\right] \int_{c\frac{65}{64}\re z_n}^{\frac{65}{64}\re z_n}\varphi(t)dt\\
&\geq \left[\frac{\re z_n}{64\pi}\right] \frac{\frac{65}{64}\re z_n}{A(E(\frac{65}{64}\re z_n))},
\end{aligned}
\end{equation*}
and
\begin{equation*}
\begin{aligned}
\dens\left(S, Q\right) &\geq \dfrac{C_0}{A\left(E^{2}(\re z_n)\right)}
\end{aligned}
\end{equation*}
for some constant $C_0 >0$. Therefore, we obtain the following local estimate as in \eqref{locale} by using Koebe's theorem as before:
\begin{equation}\label{ea}
\begin{aligned}
\area &\left(\left(T_{n-1}\setminus T_n\right) \cap D\left(z,\tau r_{n}(z)\right)\right)\\
&\geq \dfrac{C_1}{A\left(E^{2}(\re z_n)\right)}\cdot \area D\left(z,\tau r_{n}(z)\right),
\end{aligned}
\end{equation}
where $C_1$ is some positive constant.

To spread this to the global estimate, we need the fact that $A(r)<\log r$ for large $r$. So in this way we have an estimate for $\re G(z)$:
\begin{equation*}
\begin{aligned}
\re G(z) &\leq \log M(e^{\re z}, g)\\
&\leq \exp \left(A(e^{\re z})\cdot\re z  \right)\\
&\leq \exp^{2} (\re z),
\end{aligned}
\end{equation*}
which implies that
$$\re z_n =\re G^{n}(z)\leq \exp^{2n}(\re z)=E^{2n}(\re z).$$
Now by using \eqref{ea} the global estimate, which is similar to \eqref{globaless},  is given as follows:
\begin{equation*}
\begin{aligned}
\area \left(\left(T_{n-1}\setminus T_n\right) \cap P_{n-1}\right)&\geq \sum_{z\in L_{n}}\area \left(\left(T_{n-1} \setminus T_{n}\right) \cap D(z, \tau r_{n}(z))\right)\\
&\geq \sum_{z\in L_{n}} \dfrac{C_1}{A\left(E^{2}(\re z_n)\right)}\cdot \area D(z, \tau r_{n}(z))\\
&\geq \dfrac{C_1}{A\left(E^{2n+2}\left(\frac{65}{64}\re w\right)\right)} \sum_{z\in L_{n}} \area D(z, \tau r_{n}(z))\\
&\geq \frac{1}{16}\dfrac{C_1}{A\left(E^{2n+2}\left(\frac{65}{64}\re w\right)\right)}\cdot \area \left(T_{n}\cap P_n\right).
\end{aligned}
\end{equation*}

This, as in \eqref{infinite product}, implies that
\begin{equation*}
\begin{aligned}
\area\left(T\cap P_{\infty}\right)\leq \prod_{n=1}^{\infty}\frac{1}{1+ \dfrac{1}{16}\dfrac{C_1}{A\left(E^{2n+2}\left(\frac{65}{64}\re w\right)\right)}}\cdot \area\left(T_1 \cap P_1\right).
\end{aligned}
\end{equation*}
Since the function $A(r)$ is increasing, we can use similar argument as \eqref{compart2} to deduce from condition \eqref{di} that
\begin{equation*}
\sum_{n=1}^{\infty}\dfrac{1}{A\left(E^{2n+2}\left(\frac{65}{64}\re w\right)\right)} = \infty.
\end{equation*}
The rest is the same as before and we omit details here. Therefore, we see that $\area \j(g)=0$ for the disjoint type function $g$.

We still need to transfer this result to our original function $f$. This is done as in the proof of Theorem \ref{thm1.2}, using Theorem \ref{main theorem}. We omit details here. Therefore, we have the conclusion that $\area \i(f)=0$.


\end{proof}


\begin{proof}[Proof of Theorem \ref{infinite-order1}]

We define
$$g(z)=\frac{1}{f(z)-a}.$$
Since $f(z)$ has a finite direct singular value $a$, then $g(z)$ has a direct tract over infinity. Choose $\varepsilon>0$ small enough, and denote
$$\alpha(r):=\meas \left\{t\in [0, 2\pi]: \left|f(re^{it})-a\right|<\varepsilon\right\}.$$
We will follow the notations given in the introduction with respect to $f$. Recall that $\theta(r)=\meas\left\{t\in[0,2\pi]: \left|f(re^{it})\right|<e^R\right\}$, which is defined in \eqref{theta1}. By choosing $R>0$ large enough we can have $e^R > |a|+\varepsilon$. Therefore, we see that
\begin{equation}\label{comp}
\theta(r)\geq \alpha(r)
\end{equation}
holds for large $r$. On the other hand, $\alpha(r)$ also denotes the linear measure of the set $U_{0}:=\left\{t\in[0, 2\pi]: |g(re^{it})|>1/\varepsilon\right\}$. Applying Lemma \ref{tsuji} to the function $g$, with $G=U_{0}$ and $\beta(r)=\alpha(r)$ we see that
\begin{equation}
\log\log M_{U_{0}}(r,g)\geq \pi \int_{r_1}^{\alpha_0 r}\frac{dt}{t\cdot \alpha(t)}+ \mathcal{O}(1),
\end{equation}
where $r_1 >0$ and $0<\alpha_0<1$. This, together with (\ref{comp}), implies that
\begin{equation}
\begin{aligned}
\log\log M_{U_{0}}(r,g) &\geq \pi \int_{r_1}^{\alpha_0 r}\frac{dt}{t\cdot \theta(t)}+ \mathcal{O}(1)\\
&\geq \pi \int_{(\alpha_0 r)^c}^{\alpha_0 r}\frac{dt}{t\cdot \theta(t)}+ \mathcal{O}(1),
\end{aligned}
\end{equation}
here $c$ satisfies the condition from Theorem \ref{infinite-order2}.

Then by applying Lemma \ref{maxi}, Lemma \ref{nevan} and Lemma \ref{maxag}, and by taking $R=2r$ we have
\begin{equation*}
\begin{aligned}
\log M_{U_{0}}(r, g)&\leq 3T(2r,g)+\mathcal{O}(1)\\
&= 3 T(2r,f)+\mathcal{O}(1)\\
&\leq  3 \log M(2r,f)+\mathcal{O}(1).
\end{aligned}
\end{equation*}
Therefore, we get
\begin{equation}\label{cmpa}
\begin{aligned}
\pi \int_{(\alpha_0 r)^c}^{\alpha_0 r}\frac{dt}{t\cdot \theta(t)} &\leq \log \log M(2r,f)+ \mathcal{O}(1)\\
&\leq A(2r)\log(2r) +\mathcal{O}(1).
\end{aligned}
\end{equation}
By the Cauchy-Schwarz inequality,
\begin{equation*}
\begin{aligned}
\left(\log (\alpha_0 r)^{1-c} \right)^2 &= \left(\int_{(\alpha_0 r)^c}^{\alpha_0 r}\frac{dt}{t}\right)^2 =\left(\int_{(\alpha_0 r)^c}^{\alpha_0 r} \sqrt{\frac{\theta(t)}{t}}\frac{1}{\sqrt{\theta(t)t}}dt  \right)^2 \\
&\leq \int_{(\alpha_0 r)^c}^{\alpha_0 r} \theta(t) \frac{dt}{t} \cdot \int_{(\alpha_0 r)^c}^{\alpha_0 r} \frac{dt}{\theta(t)\cdot t}.
\end{aligned}
\end{equation*}
Together with (\ref{cmpa}) we see that
\begin{equation}\label{require}
\frac{1}{\log r}\int_{(\alpha_0 r)^c}^{\alpha_0 r} \theta(t) \frac{dt}{t} \geq \frac{C_0}{A(2r)}
\end{equation}
for large $r$, where $C_0>0$ is some constant. Thus, we obtain from \eqref{require} that
\begin{equation}\label{require11}
\begin{aligned}
\frac{1}{\log r}\int_{r^c}^{r} \theta(t) \frac{dt}{t} &\geq \frac{1}{\log (r/\alpha_0)}\int_{r^c}^{r} \theta(t) \frac{dt}{t}\\
&\geq \frac{C_0}{A(2r/\alpha_0)}:=\frac{1}{A_0(r)}.
\end{aligned}
\end{equation}
Thus, \eqref{didi} is satisfied with $A(r)$ replaced by $A_0 (r)$. Moreover, by using \eqref{di} it is easy to show that
\begin{equation}
\sum_{k=1}^{\infty}\frac{1}{A_0(E^{k}(0))}=\infty.
\end{equation}

Therefore, we can apply Theorem \ref{infinite-order2} to obtain that $\area\i(f)=0$.

\end{proof}

\begin{ack}
The author would like to thank his Ph.D. supervisor Walter Bergweiler for his constant help and guidance, and Lasse Rempe-Gillen for his valuable suggestions. I also thank the referee for helpful comments. I would also like to thank the China Scholarship Council for support.
\end{ack}


\begin{thebibliography}{BFRG15}

\bibitem[AB12]{aspenberg1}
M.~Aspenberg and W.~Bergweiler, \emph{Entire functions with {J}ulia sets of
  positive measure}, Math. Ann. \textbf{352~} (2012), no.~1, 27--54.

\bibitem[Ahl06]{ahlfors8}
L.~V. Ahlfors, \emph{Lectures on quasiconformal mappings}, second ed.,
  University Lecture Series, vol.~38, American Mathematical Society,
  Providence, RI, 2006.

\bibitem[BC16]{bergweiler17}
W.~Bergweiler and I.~Chyzhykov, \emph{Lebesgue measure of escaping sets of
  entire functions of completely regular growth}, J. London Math. Soc. (2)
  \textbf{94} (2016), no.~2, 639--661.

\bibitem[Ber93]{bergweiler1}
W.~Bergweiler, \emph{Iteration of meromorphic functions}, Bull. Amer. Math.
  Soc. \textbf{29} (1993), no.~2, 151--188.

\bibitem[BFRG15]{bergweiler6}
W.~Bergweiler, N.~Fagella, and L.~Rempe-Gillen, \emph{Hyperbolic entire
  functions with bounded {Fatou} components}, Comment. Math. Helv. \textbf{90}
  (2015), no.~4, 799--829.

\bibitem[BRS08]{bergweiler3}
W.~Bergweiler, P.~J. Rippon, and G.~M. Stallard, \emph{Dynamics of meromorphic
  functions with direct or logarithmic singularities}, Proc. London Math. Soc.
  (3) \textbf{97} (2008), no.~2, 368--400.

\bibitem[EL92]{eremenko2}
A.~Eremenko and M.~Lyubich, \emph{Dynamical properties of some classes of
  entire functions}, Ann. Inst. Fourier \textbf{42} (1992), no.~4, 989--1020.

\bibitem[Ere89]{eremenko3}
A.~Eremenko, \emph{On the iteration of entire functions}, Dynamical systems and
  ergodic theory ({W}arsaw, 1986), Banach Center Publ., vol.~23, PWN, Warsaw,
  1989, pp.~339--345.

\bibitem[Fal03]{falconer1}
K.~J. Falconer, \emph{Fractal geometry: Mathematical foundations and
  applications}, second ed., John Wiley \& Sons, Inc., Hoboken, NJ, 2003.

\bibitem[Fat26]{fatou4}
P.~Fatou, \emph{Sur l'it\'eration des fonctions transcendantes enti\`eres},
  Acta Math. \textbf{47} (1926), no.~4, 337--370.

\bibitem[GO08]{goldbergmero}
A.~A. Goldberg and I.~V. Ostrovskii, \emph{Value distribution of meromorphic
  functions}, Translations of Mathematical Monographs, vol. 236, American
  Mathematical Society, Providence, RI, 2008.

\bibitem[Hay64]{hayman1}
W.~K. Hayman, \emph{Meromorphic functions}, Oxford Mathematical Monographs,
  Clarendon Press, Oxford, 1964.

\bibitem[McM87]{mcmullen11}
C.~T. McMullen, \emph{Area and{ Hausdorff dimension of Julia} sets of entire
  functions}, Trans. Amer. Math. Soc. \textbf{300} (1987), no.~1, 329--342.

\bibitem[Mil06]{milnor10}
J.~Milnor, \emph{Dynamics in one complex variable}, Annals of Mathemtaics
  Studies, vol. 160, Princeton University Press, Princeton, NJ, 2006.

\bibitem[Mis81]{misiurewicz2}
M.~Misiurewicz, \emph{On iterates of {$e^{z}$}}, Ergodic Theory Dynam. Systems
  \textbf{1} (1981), no.~1, 103--106.

\bibitem[Nev70]{nevanlinna}
R.~Nevanlinna, \emph{Analytic functions}, Translated from the second German
  edition by Phillip Emig. Die Grundlehren der mathematischen Wissenschaften,
  Band 162, Springer-Verlag, New York-Berlin, 1970.

\bibitem[Pom92]{pommerenke1}
Ch. Pommerenke, \emph{Boundary behaviour of conformal maps}, Grundlehren der
  mathematischen Wissenschaften, vol. 299, Springer-Verlag, Berlin, 1992.

\bibitem[Rem09]{rempe8}
L.~Rempe, \emph{Rigidity of escaping dynamics for transcendental entire
  functions}, Acta Math. \textbf{203} (2009), no.~2, 235--267.

\bibitem[RG16]{rempe9}
L.~Rempe-Gillen, \emph{Arc-like continua, {Julia sets of entire functions, and
  Eremenko's Conjecture}}, arXiv:1610.06278 (2016).

\bibitem[RGS17]{rempe16}
L.~Rempe-Gillen and D.~J. Sixsmith, \emph{Hyperbolic entire functions and the
  {Eremenko-Lyubich} class: {Class} $\mathcal{B}$ or not class $\mathcal{B}$?},
  Math. Z. \textbf{286} (2017), no.~3--4, 783--800.

\bibitem[RRRS11]{rempe2}
G.~Rottenfusser, J.~R\"uckert, L.~Rempe, and D.~Schleicher, \emph{Dynamic rays
  of bounded-type entire functions}, Ann. of Math. (2) \textbf{173} (2011),
  no.~1, 77--125.

\bibitem[Sch10]{schleicher3}
D.~Schleicher, \emph{Dynamics of entire functions}, Holomorphic dynamical
  systems, Lecture Notes in Math., vol. 1998, Springer, Berlin, 2010,
  pp.~295--339.

\bibitem[Tsu75]{tsuji1975potential}
M.~Tsuji, \emph{Potential theory in modern function theory}, Chelsea Publishing
  Co., New York, 1975.

\end{thebibliography}

\end{document}